\documentclass[a4paper,11pt,twoside]{amsart}
\usepackage[english]{babel}
\usepackage[utf8]{inputenc}

\usepackage[a4paper,inner=3cm,outer=3cm,top=4cm,bottom=4cm,pdftex]{geometry}
\usepackage{fancyhdr}
\usepackage{fancyhdr}
\usepackage{titlesec}
\usepackage{enumerate}
\usepackage{color}
\usepackage{bold-extra}
\titleformat{\section}{\normalfont\scshape\centering}{\thesection}{1em}{}
\titleformat{\subsection}{\bfseries}{\thesubsection}{1em}{}

\usepackage{color}
\usepackage{bold-extra}
\usepackage{ mathrsfs }

\usepackage{comment}
\usepackage{graphics}
\usepackage{aliascnt}
\usepackage[pdftex,citecolor=green,linkcolor=red]{hyperref}

\usepackage{amsmath}
\usepackage{amsfonts}
\usepackage{amssymb}
\usepackage{amsthm}
\usepackage{comment}
\usepackage{mathtools}

\newtheorem{theorem}{Theorem}[section]

\newtheorem{lemma}[theorem]{Lemma}
\newtheorem{proposition}[theorem]{Proposition}
\theoremstyle{definition}

\newtheorem{remark}[theorem]{Remark}
\newtheorem{conjecture}[theorem]{Conjecture}
\numberwithin{equation}{section}

\renewcommand{\Re}{\textnormal{Re}}
\renewcommand{\Im}{\textnormal{Im}}
\renewcommand\d{\textnormal{d}}

\begin{document}

\title{A note on zero density results implying large value estimates for Dirichlet polynomials}

\author{Kaisa Matom\"aki}
\address{Department of Mathematics and Statistics, University of Turku, 20014 Turku, Finland}
\email{ksmato@utu.fi}

\author{Joni Ter\"{a}v\"{a}inen}
\address{Department of Mathematics and Statistics, University of Turku, 20014 Turku, Finland}
\email{joni.p.teravainen@gmail.com}

\begin{abstract}
In this note we investigate connections between zero density estimates for the Riemann zeta function and large value estimates for Dirichlet polynomials. It is well known that estimates of the latter type imply estimates of the former type. Our goal is to show that there is an implication to the other direction as well, i.e. zero density estimates for the Riemann zeta function imply large value estimates for Dirichlet polynomials.
\end{abstract}

\maketitle

\section{Introduction}

In this note we investigate connections between zero density estimates for the Riemann zeta function and large value estimates for Dirichlet polynomials. It is well known that estimates of the latter type imply estimates of the former type. Our goal is to show that there is an implication to the other direction as well, i.e. zero density estimates for the Riemann zeta function imply large value estimates for Dirichlet polynomials.

To discuss this in more detail, we need to introduce some notation. For $\sigma \in [1/2, 1]$ and $T \geq 0$, let
\[
N(\sigma, T) \coloneqq \#\{\rho = \beta + i\gamma \colon \zeta(\rho) = 0, \beta \geq \sigma, |\gamma| \leq T\},
\]
where the zeta zeros are counted with multiplicities. As usual, by zero density estimates, we mean upper bounds for $N(\sigma, T)$. While the Riemann hypothesis states that $N(\sigma,T)=0$ for every $\sigma>1/2$, in many applications of the theory of the zeros of the zeta function (such as study of primes in short intervals) also weaker bounds for $N(\sigma, T)$ are useful. 

The usual way to obtain results concerning $N(\sigma, T)$ is to use zero-detecting polynomials, i.e. to show that if there are several zeros of the Riemann zeta function, then there exists a Dirichlet polynomial that takes several large values. After this reduction one can utilize known large value results for Dirichlet polynomials. In Section~\ref{se:converse} we show the following reduction from zeroes to large values of Dirichlet polynomials.

\begin{proposition}
\label{prop:converse}
Let $\varepsilon >0, \nu \in (0, 1/2],$ and $T \geq 3$. Let
\[
\mathcal{T} \coloneqq \{\rho = \beta + i \gamma \colon \zeta(\rho) = 0, \beta \geq 1-\nu, |\gamma| \leq T\}.
\] 
Then we can partition $\mathcal{T} = \mathcal{T}_1 \cup \mathcal{T}_2$ in such a way that $\# \mathcal{T}_1 \ll_\varepsilon T^{2\nu+2\varepsilon}$ and, for each zeta zero $\rho = \beta + i\gamma \in \mathcal{T}_2$, there exists $M \in [T^\varepsilon, T^{1/2}/2]$ and $M' \in (M, 2M]$ such that
\[
\left|\sum_{M<m\leq M'} \frac{1}{m^{1+i\gamma}}\right| \geq M^{-\nu-\varepsilon}.
\]
\end{proposition}

Proposition~\ref{prop:converse} implies that $N(1-\nu, T)$ is bounded by $O(T^{2\nu+2\varepsilon})$ plus the number of one-spaced\footnote{We say that a set is one-spaced if any two elements are at least distance $1$ apart.} $t \in [-T, T]$ such that 
\[
\left|\sum_{M<m\leq M'} \frac{1}{m^{1+i t}}\right| \geq M^{-\nu-\varepsilon}
\]
for some $M \in [T^\varepsilon, T^{1/2}]$ and $M' \in (M, 2M]$.

The main aim of this note is to show a converse result.
\begin{theorem}[Zero density results imply large value estimates]
\label{th:density->large}
Let $\varepsilon > 0, \nu \in [0, 1/2],$ and $T\geq 1$. There exists a constant $C = C(\varepsilon)$ such that
\begin{align}\begin{split}\label{eq:mainthm}
&\left|\left\{t \in [-T, T] \colon \left|\sum_{M <m\leq M'} \frac{1}{m^{1+it}}\right| \geq M^{-\nu} \text{  for some $M \in [T^\varepsilon, T^{1/2}/2]$ and $M' \in (M, 2M]$}\right\}\right|\\
&\quad \ll_{\varepsilon} T^\varepsilon \max_{1-\nu-\varepsilon \leq \alpha \leq 1} T^{\frac{\alpha-(1-\nu)}{2}} N(\alpha, C\cdot T) + T^{\frac{\nu}{2} + \varepsilon}.
\end{split}
\end{align}
\end{theorem}

\begin{remark}
Note that we can get a similar result with $T^{2\nu + 2\varepsilon}$ in place of $T^{\frac{\nu}{2} + \varepsilon}$ e.g. for the M\"obius and prime polynomials using Heath-Brown's identity. Indeed applying Heath-Brown's identity (see e.g.~\cite[Proposition 13.3 and Exercise 1 following it]{iw-kow} with $K = \lfloor 100/\varepsilon\rfloor$), and denoting by $F(s)$ either of the polynomials $\sum_{M<m\leq M'} \mu(m)/m^s$ or $\sum_{M<m\leq M'} \Lambda(m)/m^s$, then $F(s)$ can be morally decomposed into products of polynomials of the type $\sum_{L < \ell \leq L'} 1/\ell^s$ with $L'\in [L,2L]$, $L\in [T^{\varepsilon},T^{1/2}]$ and polynomials of length $\leq T^{\varepsilon}$. Now if $|F(s)|$ is large, then either one of $|\sum_{L < \ell \leq L'} 1/\ell^s|$ with $L \in [T^\varepsilon, T^{1/2}]$ is large or there is a short polynomial of length $\in [T^{\varepsilon^2/100}, T^{\varepsilon}]$ with large value. In the latter case the number of large values can be seen to be $\ll T^{2\nu+2\varepsilon}$ using the discrete mean value theorem for Dirichlet polynomials to a suitable power of the polynomial similarly to~\eqref{eq:powerDir} below.
\end{remark}

Typically estimates $N'(\sigma, T)$ for $N(\sigma,T)$ weaken with slope $> 1$ when $\sigma$ decreases and in such case the maximum in~\eqref{eq:mainthm} is attained at $\alpha=1-\nu-\varepsilon$, and hence typically Theorem~\ref{th:density->large} gives an estimate of the form
\begin{align*}
&\left|\left\{t \in [-T, T] \colon \left|\sum_{M <m\leq M'} \frac{1}{m^{1+it}}\right| \geq M^{-\nu} \text{  for some $M \in [T^\varepsilon, T^{1/2}/2]$ and $M' \in (M, 2M]$}\right\}\right|\\ 
&\ll_{\varepsilon} T^\varepsilon N'(1-\nu-\varepsilon,C\cdot T) + T^{\frac{\nu}{2}+\varepsilon}
\end{align*}
which corresponds to the aforementioned consequence of Proposition~\ref{prop:converse}.

Considering the case $M = T^{1/2}$ and $|t| \leq M^{\nu}/1000$, we see that the second term on the right-hand side is needed apart from $\varepsilon$ in the exponent.

Thanks to Theorem~\ref{th:density->large} we see that if one was somehow able to prove new zero density results without appealing to large value theorems, one would still obtain new information about large values as well. 

The density hypothesis (which would imply quite similar consequences as the Riemann hypothesis) states that 
\begin{align}\label{eq:DH}
N(\sigma, T) \ll_{\varepsilon} T^{2(1-\sigma)+\varepsilon}\quad \textnormal{  for every }\quad  \varepsilon>0,\, \sigma \in [1/2, 1]
\end{align}
(sometimes the factor $T^\varepsilon$ is replaced by $\log T$). In particular, if one assumes the density hypothesis~\eqref{eq:DH}, Theorem~\ref{th:density->large} implies that, for any $\varepsilon > 0$ and $\nu \in [0, 1/2]$,
\[
|\{t \in [-T, T] \colon |M(1+it)| \geq M^{-\nu}\}| \ll_{\varepsilon} T^{2\nu+\varepsilon}.
\]
For some applications, it would be important to slightly beat the density hypothesis. We write such a strengthening in the following conjecture.
\begin{conjecture}[Stronger density hypothesis] \label{conj:StrongDH}
Let $\varepsilon > 0$ . There exists $\delta = \delta(\varepsilon) \in (0,1)$ such that, for any $\nu \in [0, 1/2-\varepsilon)$ and $T \geq 1$, one has
\begin{equation}
\label{eq:densHyp}
N(1-\nu, T) \ll_{\varepsilon} T^{(2-\delta)\nu}.
\end{equation}
\end{conjecture}
Note that by work of Bourgain~\cite{Bourgain} (refining Jutila's work~\cite{Jutila}), Conjecture~\ref{conj:StrongDH} is known to hold for $\varepsilon = 25/32$.

We have the following application of Theorem~\ref{th:density->large}.
\begin{theorem}\label{thm_DH}
Assume that Conjecture~\ref{conj:StrongDH} holds. Let $\varepsilon > 0$. Let $X$ be large enough in terms of $\varepsilon$,  and let $h = (\log X)^{2+\varepsilon}$. Then, for all but $o_{X\to \infty}(X)$ integers $x \in [1, X]$, we have
\begin{align*}
|\{p_1p_2\in (x,x+h] \colon (\log X)^{1+\varepsilon/2} \leq p_1\leq (\log X)^{1+\varepsilon} \}| \gg_\varepsilon \frac{h}{\log X}.    
\end{align*}
\end{theorem}

This proves conditionally on the stronger density hypothesis the existence of $E_2$ numbers (products of exactly two primes) in almost all intervals of length $(\log X)^{2+\varepsilon}$ around $X$. Previous work of Harman~\cite{harman-almostprimes} implied this under the same Conjecture~\ref{conj:StrongDH} for intervals of length $(\log X)^{3+\varepsilon}$ around $X$. One can easily deduce Theorem~\ref{thm_DH} from Proposition~\ref{prop_sigma1/2} below using our recent work~\cite{matomaki-teravainenE2} that works unconditionally for intervals of length $(\log X)^{2.1}$ around $X$; we sketch this implication in Section~\ref{sec:thm_DH}. We also note that, under the stronger assumption of the Riemann hypothesis, it is known by work of Selberg~\cite{Selberg43} that almost all intervals of length $(\log X)^{2+\varepsilon}$ around $X$ contain primes (and thus also $E_2$ numbers).

For more information about known bounds for $N(\sigma, T)$ and related matters, see e.g.~\cite[Chapter 10]{iw-kow}.

\section{Outline of the proof of Theorem~\ref{th:density->large}}
We outline the proof of Theorem~\ref{th:density->large} via stating lemmas whose proofs we postpone. We shall use the notation
\begin{align}\label{eq:RT}
R_{\sigma,\eta}(T)\coloneqq \left|\left\{t\in [-T,T]\colon\,\, \max_{\substack{1\leq A\leq B\leq T^{1/2}\\B\leq 2A}} \left|\sum_{A < n \leq B} \frac{1}{n^{\sigma +it}}\right|\geq T^\eta\right\}\right|.    
\end{align}
Note that trivially $R_{\sigma, \eta}(T) \leq R_{\sigma, \eta'}(T)$ for any $\eta' \leq \eta$. Our first step is to use partial summation to move the claim~\eqref{eq:mainthm} to the line $\sigma = 1-\nu-\eta$:

\begin{lemma}\label{le_density1} Let $\nu \in [0, 1/2]$ and $\varepsilon, \eta \in (0, 1)$. Let $T \geq 4^{2/(\varepsilon \eta)}$. Then
\begin{align*}
&\left|\left\{t \in [-T, T]\colon\left|\sum_{M <m\leq M'} \frac{1}{m^{1+it}}\right| \geq M^{-\nu} \text{  for some $M \in [T^\varepsilon, T^{1/2}/2]$ and $M' \in (M, 2M]$}\right\}\right| \\
& \leq R_{1-\nu-\eta,\varepsilon \eta/2}(T).   
\end{align*}
\end{lemma}

Next we show that if there are many large values of Dirichlet polynomials (i.e. $R_{\sigma,\eta}(T)$ is large) then there are many large values of the Riemann zeta function.

\begin{lemma}\label{le_FromRtoLargeZeta}
For every $\eta_0 \in (0, 1)$ there exists $T_0 = T_0(\eta_0)$ such that the following holds. For every $\sigma \in [1/2, 1]$, $\eta \in (\eta_0, 1)$ and $T \geq T_0$ either $R_{\sigma, \eta}(T) \leq 4T^{(1-\sigma)/2}$ or there exist $\beta \geq \eta - 3\frac{\log \log T}{\log T}$ such that
\begin{align*}
\left|\left\{t \in [-2T, 2T] \colon \left|\zeta(\sigma+\tfrac{1}{\log T}+it)\right|\in (T^{\beta},2T^{\beta}]\right\} \setminus [-10, 10]\right| \geq \frac{R_{\sigma,\eta}(T)T^{\eta}}{50 T^\beta (\log T)^{2}}.
\end{align*}
\end{lemma}

We will utilize this in two different ways. First, we show the counter-intuitive fact that if the Riemann zeta function is large at some point, it must have a zero nearby. 

\begin{lemma}[Large value of $\zeta(s)$ implies the existence of a nearby zero]\label{le_density3} 
Let $T\geq T_0$ for a large enough constant $T_0$. Let $\sigma \in [1/2, 1]$. 
\begin{enumerate}[(i)]
\item Suppose that 
\begin{align}\label{eq:zetalarge}
|\zeta(\sigma+it)|\geq T^{1/(\log \log T)^{100}}    
\end{align}
for some $t\in \mathbb{R}$ with $|t| \in [(\log T)^2/2, T]$. Then there is a zero $\rho$ of the Riemann zeta function in the rectangle
\begin{align*}
  \textnormal{Re}(\rho)\geq \sigma-\frac{1}{(\log \log T)^{1/2}},\quad |\textnormal{Im}(\rho)- t| \leq (\log T)^2/4.
\end{align*}
\item Let $\beta > 0$. If $T \geq e^{e^{1/\beta}}$, then
\[
|\{t \in [-T, T] \colon |\zeta(\sigma+it)| \geq T^\beta\}| \leq \left(N\left(\sigma-\frac{1}{(\log \log T)^{1/2}}, 2T\right) + 1\right) \cdot (\log T)^2. 
\]
\end{enumerate}
\end{lemma}

The previous lemma together with Lemma~\ref{le_FromRtoLargeZeta} will allow us at the end to relate the number of large values to the number of zeros. However, it is most useful when the parameter $\beta$ we obtain from Lemma~\ref{le_FromRtoLargeZeta} is small as we lose a factor $T^\beta$ in the measure of the set we obtain from Lemma~\ref{le_FromRtoLargeZeta}. For this reason we need to set up an inductive argument, where we also use that if there are many large values of the zeta function then there are many large values of Dirichlet polynomials.

\begin{lemma}
\label{le_zetabig->Rbig}
Let $\beta_0, \varepsilon_0 > 0$. There exists $T_0 = T_0(\beta_0, \varepsilon_0)$ such that the following holds. Let $\sigma \in [1/2, 1], \beta > \beta_0$, and $\varepsilon > \varepsilon_0$. Then
\[
|\{t \in [-T, T] \colon |\zeta(\sigma+it)| \geq T^\beta\} \setminus [-2\pi, 2\pi]| \leq R_{\sigma + (2-\varepsilon)\beta, \varepsilon \beta/3}(2T). 
\]
\end{lemma}

Combining Lemmas~\ref{le_FromRtoLargeZeta},~\ref{le_density3} and~\ref{le_zetabig->Rbig}, we immediately obtain the following lemma.

\begin{lemma}
\label{le:RsigmaRec}
Let $\eta_0, \varepsilon_0 > 0$. There exists $T_0 = T_0(\eta_0, \varepsilon_0)$ such that the following holds. Let $\varepsilon > \varepsilon_0, \eta > \eta_0, \sigma \in [1/2, 1]$, and $T \geq T_0$. Then there exists $\beta \geq \eta-3\log\log T/\log T$ such that both of the following hold.
\begin{enumerate}[(i)]
\item We have
\begin{align}\label{eq:R2}
R_{\sigma,\eta}(T)\leq 50 T^{\beta-\eta}(\log T)^{2} R_{\sigma+(2-\varepsilon)\beta+1/\log T,\varepsilon \beta/4}(4T) + 4 T^{\frac{1-\sigma}{2}}.
\end{align}
\item We have
\begin{align}\label{eq:R3}
R_{\sigma,\eta}(T)\leq 50 T^{\beta-\eta}(\log T)^{4}\left(N\left(\sigma-\tfrac{1}{(\log \log T)^{1/2}},4T\right)+1\right) + 4 T^{\frac{1-\sigma}{2}}.
\end{align}
\end{enumerate}
\end{lemma}

Finally using Lemma~\ref{le:RsigmaRec} an inductive argument will yield
\begin{lemma}\label{le_density2} Let $\varepsilon, \eta > 0$. There exists $C = C(\varepsilon)$ such that the following holds. Let $T \geq 1$. Then, for any $\sigma \in [1/2, 1]$, we have
\begin{align}
\label{eq:indconcl}
R_{\sigma,\eta}(T)\ll_{\varepsilon, \eta} T^{2\varepsilon} \max_{\sigma-\varepsilon \leq \alpha \leq 1} T^{(\alpha-\sigma)/2} N(\alpha, C \cdot T) + T^{(1-\sigma)/2 + 2\varepsilon}.
\end{align}
\end{lemma}

\begin{proof}[Proof of Theorem~\ref{th:density->large} assuming Lemmas~\ref{le_density1} and~\ref{le_density2}] By Lemma~\ref{le_density1} with $\eta = \varepsilon/2$, we have
\begin{align*}
&\left|\left\{t \in [-T, T]\colon\left|\sum_{M <m\leq M'} \frac{1}{m^{1+it}}\right| \geq M^{-\nu} \text{  for some $M \in [T^\varepsilon, T^{1/2}/2]$ and $M' \in (M, 2M]$}\right\}\right| \\
&\leq R_{1-\nu-\varepsilon/2, \varepsilon^2/4}(T).
\end{align*}
The claim now follows from Lemma~\ref{le_density2} with $\varepsilon/4$ in place of $\varepsilon$.
\end{proof}
 
\section{Auxiliary results}
In this section we collect a few standard lemmas concerning the Riemann zeta function and its zeros. The first one contains two forms of the approximate functional equation.
\begin{lemma}
\label{le:AFE}
\begin{enumerate}[(i)]
\item Let $\sigma \in [0, 1], t \in \mathbb{R},$ and let $x, y \geq 1$ be such that $2\pi xy = |t|$. Then
\begin{align*}
\zeta(\sigma+it)=\sum_{n\leq x} \frac{1}{n^{\sigma+it}}+ \chi(\sigma+it)\sum_{n\leq y} \frac{1}{n^{1-\sigma-it}}+O\left(\frac{\log |t|}{x^{\sigma}} + \frac{x^{1-\sigma}}{|t|^{1/2}}\right),  \end{align*}
where $\chi\colon \mathbb{C} \to \mathbb{C}$ satisfies 
\begin{align}
\label{eq:chibound}
|\chi(\sigma+it)|\leq 100 (|t|+1)^{1/2-\sigma}.
\end{align}
\item Let $T \geq 3, \sigma \geq 1/2,$ and $t \in [T, 2T]$. Then
\[
\zeta(\sigma+it) = \sum_{n \leq T} \frac{1}{n^{\sigma + it}} + O(T^{-\sigma}).
\]
\end{enumerate}
\end{lemma}
\begin{proof} Note that in part (i) we have $|t| \geq 2\pi$ and in part (ii) we have $|t| \geq 3$, so there are no issues with the pole of $\zeta(s)$ at $s = 1$. Part (i) follows e.g. from \cite[Theorem 4.13, (4.12.3) and the formulas preceding (4.12.3)]{titchmarsh} (note that while~\cite{titchmarsh} states the result for $\sigma \in (0,1)$, it actually holds uniformly in our wider range, see e.g. the original result of Hardy and Littlewood~\cite[Theorem I]{H-L}) whereas part (ii) follows e.g. from~\cite[Theorem 4.11]{titchmarsh}.
\end{proof} 

The second lemma gives a convexity bound for $\zeta(s)$.
\begin{lemma}
\label{le:convexity}
Let $\sigma \in [0, 1]$ and $t \in \mathbb{R}$ with $|t| \geq 2\pi$.Then
\begin{align*}
|\zeta(\sigma+it)| \ll |t|^{(1-\sigma)/2} \log |t|.
\end{align*}
\end{lemma}
\begin{proof} 
This follows from applying Lemma~\ref{le:AFE}(i) with $x = y = |t/(2\pi)|^{1/2}$ and estimating all terms trivially.
\end{proof}

The third lemma gives an upper bound for the number of zeros in boxes of height $1$.
\begin{lemma}
\label{le:zerosinboxes} 
Let $U\in \mathbb{R}$. Then there are $\ll \log(|U|+2)$ zeros $\rho$ of $\zeta$ (counted with multiplicities) in the region $\Re(\rho)\geq 0, \Im(\rho)\in [U,U+1]$. 
\end{lemma}
\begin{proof} 
See e.g.~\cite[Theorem 9.2]{titchmarsh}.
\end{proof}

\section{Partial summation and proof of Lemma~\ref{le_density1}}
We shall frequently use the following simple partial summation lemma. 
\begin{lemma}[Comparing zeta sums on different lines]\label{le_partialsummation}
Let $\sigma, t, \beta\in \mathbb{R}$. Let $M_2 > M_1 \geq 1.$ If $\beta\geq 0$, we have
\begin{align*}
\left|\sum_{M_1 < m \leq M_2}\frac{1}{m^{\sigma+it}}\right|\leq 4 M_2^{\beta}\max_{M_1 < y \leq M_2} \left|\sum_{M_1 < m \leq y} \frac{1}{m^{\sigma+\beta+it}}\right|,
\end{align*}
and if $\beta<0$, we have 
\begin{align*}
\left|\sum_{M_1 < m \leq M_2}\frac{1}{m^{\sigma+it}}\right|\leq 4 M_1^{\beta}\max_{M_1 < y \leq M_2} \left|\sum_{M_1 < m \leq y} \frac{1}{m^{\sigma+\beta+it}}\right|.
\end{align*}
\end{lemma}

\begin{proof}
By partial summation, we have
\begin{align*}
\sum_{M_1 < m \leq M_2}\frac{1}{m^{\sigma+it}} &= M_2^{\beta} \sum_{M_1 < m \leq M_2} \frac{1}{m^{\sigma+\beta+it}} -\beta\int_{M_1}^{M_2} y^{\beta-1} \sum_{M_1 < m \leq y} \frac{1}{m^{\sigma+\beta+it}} \d y, 
\end{align*}
and hence by the triangle inequality we obtain
\begin{align*}
\left|\sum_{M_1 < m \leq M_2}\frac{1}{m^{\sigma+it}}\right|\leq 2 (M_2^{\beta}+M_1^{\beta}) \cdot \max_{M_1< y\leq M_2} \left|\sum_{M_1 < m \leq y} \frac{1}{m^{\sigma+\beta+it}}\right|.   
\end{align*}
Both claims of the lemma follow immediately from this. 
\end{proof}

\begin{proof}[Proof of Lemma \ref{le_density1}]
By Lemma~\ref{le_partialsummation}, we have, for any $M \in [T^\varepsilon, T^{1/2}]$, $M' \in (M, 2M]$, and $|t| \leq T$,
\begin{align*}
\left|\sum_{M < m \leq M'} \frac{1}{m^{1+it}}\right| \leq 4 M^{-\nu-\eta} \max_{M \leq y \leq M'} \left|\sum_{M < m\leq y} \frac{1}{m^{1-\nu-\eta+it}}\right|.
\end{align*}
Hence, whenever we have 
\[
\left|\sum_{M < m \leq M'} \frac{1}{m^{1+it}}\right| > M^{-\nu},
\]
for some $M \in [T^\varepsilon, T^{1/2}]$ and $M' \in (M, 2M]$, we  also have
\begin{align*}
\left|\sum_{M < m\leq M'} \frac{1}{m^{1-\nu-\eta+it}}\right|> M^{\eta}/4 
\end{align*}
for some $M \in [T^\varepsilon, T^{1/2}]$ and $M' \in (M, 2M]$.
Recalling that $M \geq T^{\varepsilon}$ and $T \geq 4^{2/(\varepsilon \eta)}$, we see that $M^{\eta}/4 \geq T^{\varepsilon \eta/2}$, so the claim follows.
\end{proof}

\section{From large Dirichlet polynomial values to large zeta values --- proof of Lemma \ref{le_FromRtoLargeZeta}}
We now prove Lemma~\ref{le_FromRtoLargeZeta} using Perron's formula.
\begin{proof}[Proof of Lemma~\ref{le_FromRtoLargeZeta}]
Let $1\leq A\leq B\leq T^{1/2}$ with $B\leq 2A$. 
Consider first the case $\sigma \in [1-2/\log T, 1]$. In this case trivially
\[
\left|\sum_{A < n \leq B} \frac{1}{n^{\sigma+it}}\right| \leq \sum_{n \leq T^{1/2}} \frac{1}{n^{\sigma}} \leq T^{1/\log T} \sum_{n \leq T^{1/2}} \frac{1}{n} \leq e(\log T+1),
\]
and hence, once $T_0$ is sufficiently large in terms of $\eta_0$, we have $R_{\sigma, \eta}(T) = 0$ and the claim follows immediately.

Hence we can assume $\sigma\in [1/2,1-2/\log T]$. Let also $t\in \mathbb{R}$. By Perron's formula, we have
\begin{align*}
\sum_{A < n \leq B} \frac{1}{n^{\sigma+it}} = \frac{1}{2\pi i}\int_{1-iT}^{1+iT}\zeta(s+\sigma+it)\frac{B^s-A^s}{s}\d s+O(1).    
\end{align*}

By shifting the line of integration to $\textnormal{Re}(s)=1/\log T$, picking up the residue from the pole at $s=1-\sigma-it$ and using the convexity bound (Lemma~\ref{le:convexity}) to estimate the error from this shift, we see that
\begin{align*}
\sum_{A < n \leq B} \frac{1}{n^{\sigma+it}} = \frac{1}{2\pi i}\int_{1/\log T-iT}^{1/\log T+iT}\zeta(s+\sigma+it)\frac{B^s-A^s}{s}\d s+ \frac{B^{1-\sigma-it}-A^{1-\sigma-it}}{1-\sigma-it}+O(1).    
\end{align*}
The second term on the right is also $O(1)$, provided that $|t|\geq T^{(1-\sigma)/2}$. 

Hence, by the triangle inequality we conclude that if $|\sum_{A < n \leq B} \frac{1}{n^{\sigma+it}} |\geq T^{\eta}, |t|\geq T^{(1-\sigma)/2}$, and $T_0$ is sufficiently large in terms of $\eta_0$, then 
\begin{align*}
\int_{-T}^{T}|\zeta(\sigma+1/\log T+i(t+u))|\frac{\d u}{|u|+1/\log T}\geq \frac{T^{\eta}}{5}.   
\end{align*}
Since $|t| \geq T^{(1-\sigma)/2}$, the part with $|t+u| \leq 10$ contributes to the left hand side by the convexity bound (Lemma~\ref{le:convexity}) $\ll (\log T)^2$. Hence once $T_0$ is sufficiently large in terms of $\eta_0$, we have
\begin{align}\label{eq:R4}
\int_{-T}^{T}\mathbf{1}_{|t+u| > 10} |\zeta(\sigma+1/\log T+i(t+u))|\frac{\d u}{|u|+1/\log T}\geq \frac{T^{\eta}}{10}.   
\end{align}

By the definition of $R_{\sigma,\eta}(T)$, we have~\eqref{eq:R4} for a set of measure $\geq R_{\sigma,\eta}(T)/2$ of $t\in [-T,T] \setminus [-T^{(1-\sigma)/2}, T^{(1-\sigma)/2}]$ (unless $R_{\sigma,\eta}(T)\leq 4T^{(1-\sigma)/2}$, in which case we are done). 

Let $\mathcal{R}_{\sigma,\eta}(T)$ be the set of such $t$. Then, integrating \eqref{eq:R4} over $t\in \mathcal{R}_{\sigma,\eta}(T)$ and applying Fubini's theorem, we see that 
\begin{align*}
 \int_{-T}^{T}\int_{\mathcal{R}_{\sigma,\eta}(T)} \mathbf{1}_{|t+u| > 10} |\zeta(\sigma+1/\log T+i(t+u))|\d t\frac{\d u}{|u|+1/\log T}\geq \frac{T^{\eta}R_{\sigma,\eta}(T)}{20}.  
\end{align*}
By the pigeonhole principle, this implies that
\begin{align*}
\int_{\mathcal{R}_{\sigma,\eta}(T)}\mathbf{1}_{|t+u_0| > 10} |\zeta(\sigma+1/\log T+i(t+u_0))|\d t\geq \frac{T^{\eta} R_{\sigma,\eta}(T)}{50 \log T}     
\end{align*}
for some $|u_0|\leq T$. By dyadic decomposition, this implies that there is some $\beta\geq \eta-3\frac{\log \log T}{\log T}$ such that
\begin{align*}
 |\zeta(\sigma+1/\log T+i(t+u_0))|\in (T^{\beta},2T^{\beta}] \end{align*}
for a set of measure 
\[
\geq \frac{R_{\sigma,\eta}(T)T^{-\beta+\eta}}{50 (\log T)^{2}}
\]
of $t\in [-T,T]$ with $|t+u_0|\geq 10$. The claim follows since $|u_0| \leq T$.
\end{proof}

\section{From large zeta values to zeros of zeta --- proof of Lemma~\ref{le_density3}}
The proof of Lemma~\ref{le_density3} follows an idea of Ivi\'c \cite{ivic-density}. We will deduce Lemma~\ref{le_density3} from the following lemma which shows a similar result for the logarithmic derivative of zeta. For convenience we write the following lemma in contrapositive form compared to Lemma~\ref{le_density3}.

\begin{lemma}[Large value of $\frac{\zeta'}{\zeta}(s)$ implies the existence of a nearby zero]\label{le_logderzetalarge} Let $T \geq T_0$ for a sufficiently large constant $T_0$. Let $\sigma_1 \in [-1, 2]$ and $u \in \mathbb{R}$ with $|u| \in [2, 2T]$ be such that there are no zeroes $\rho$ of the Riemann zeta function in the rectangle
\begin{align*}
  \textnormal{Re}(\rho)\geq \sigma_1 - \frac{1}{2(\log \log T)^{1/2}},\quad |\textnormal{Im}(\rho)-u| \leq 1. 
\end{align*}
Then
\begin{align*}
\left|\frac{\zeta'}{\zeta}(\sigma_1+iu)\right| \leq C_0(\log \log T)^{1/2} \log T
\end{align*}
for some absolute constant $C_0>0$.
\end{lemma}

\begin{proof}
The partial fraction expansion of the logarithmic derivative of $\zeta$ (see e.g. \cite[Theorem 9.6(A)]{titchmarsh}) gives
\begin{align}\label{eq:zeta'}
\frac{\zeta'}{\zeta}(\sigma_1+iu)=\sum_{|\Im(\rho)-u|\leq 1}\frac{1}{\sigma_1+iu-\rho}+O(\log T)  
\end{align}
uniformly for such $\sigma_1 \in [-1, 2]$ and $u \in \mathbb{R}$ with $|u| \in [2, 2T]$. Now the claim follows using Lemma~\ref{le:zerosinboxes} and the fact that $|\sigma_1+iu-\rho|\geq 1/(2(\log \log T))^{1/2}$ for the summands of \eqref{eq:zeta'}. 
\end{proof}

\begin{proof}[Proof of Lemma~\ref{le_density3}] We first show how part (ii) follows quickly from part (i). Indeed, assuming part (i) and denoting for brevity $\mathcal{L}=(\log T)^2/4$ and $\sigma_0=\sigma-1/(\log \log T)^{1/2}$, we have
\begin{align*}
&|\{t \in [-T, T] \colon |\zeta(\sigma+it)| \geq T^\beta\}|\\
\quad &\leq \mathcal{L} \cdot \#\Bigg\{n\in \left[-\left\lceil \frac{T}{\mathcal{L}}\right\rceil, \left\lceil \frac{T}{\mathcal{L}}\right\rceil\right] \cap \mathbb{Z} \colon |\zeta(\sigma+it)| \geq T^\beta\,\, \textnormal{for some}\,\,  t\in [n\mathcal{L},(n+1)\mathcal{L}]\Bigg\}\\
\quad &\leq \mathcal{L} \cdot \#\Bigg\{n\in \left[-\left\lceil \frac{T}{\mathcal{L}}\right\rceil, \left\lceil \frac{T}{\mathcal{L}}\right\rceil\right] \cap \mathbb{Z} \colon n \geq 2 \text{ and } \zeta(\sigma'+it')=0\\
&\qquad \qquad \qquad \quad \quad \textnormal{for some}\,\, (\sigma',t)\in [\sigma_0,1]\times [(n-1)\mathcal{L},(n+2)\mathcal{L}]\Bigg\} + 3 \mathcal{L}\\
&\leq \mathcal{L} \cdot 3 \cdot N\left(\sigma-\frac{1}{(\log \log T)^{1/2}}, 2T\right) + 3 \mathcal{L},
\end{align*}
and (ii) follows.

Now it suffices to prove part (i). Let $\sigma \in [1/2, 1]$ and $t \in \mathbb{R}$ with $|t| \in [(\log T)^2/2, T]$ be such that
\begin{align}\label{eq:zetaw}
\zeta(z)\neq 0\,\,\textnormal{ for }\,\, \textnormal{Re}(z)\geq \sigma-1/(\log \log T)^{1/2},\quad |\Im(z)-t| \leq (\log T)^2/4.
\end{align}
Our goal is to show that then~\eqref{eq:zetalarge} cannot hold.

Let $s = b+it$ with $b \in [\sigma, 2]$ and $\sigma, t$ as above, and let $Y \coloneqq (\log T)^{3/2}$.  By the Mellin inversion formula, for any $n \geq 1,$ we have
\begin{align*}
e^{-n/Y}=\frac{1}{2\pi i}\int_{2-i\infty}^{2+i\infty}\Gamma(w)\left(\frac{Y}{n}\right)^{w} \d w.    
\end{align*}
Multiplying by $\Lambda(n)n^{-s}$ and summing over $n$ (and using Fubini's theorem) we obtain
\begin{align*}
\sum_{n\geq 1} \Lambda(n)e^{-n/Y}n^{-s}=-\frac{1}{2\pi i} \int_{2-i\infty}^{2+i\infty}\frac{\zeta'}{\zeta}(s+w)\Gamma(w)Y^{w}\d w.
\end{align*}
We truncate the integral at height $Y$ and use the standard bound (which follows from Stirling's formula)
\begin{align}\label{eq:gamma}
|\Gamma(w)|\ll \frac{1}{|w|\exp(|\Im(w)|)}   \qquad \text{for\,\,\,  $\Re(w) \in [-1/2, 3]$}
\end{align}
to estimate the error from this. Using also $|\frac{\zeta'}{\zeta}(z)|\leq |\frac{\zeta'}{\zeta}(2)|\ll 1$ for $\Re(z)\geq 2$, we obtain
\begin{align*}
\sum_{n\geq 1}\Lambda(n)e^{-n/Y}n^{-s}=-\frac{1}{2\pi i} \int_{2-iY}^{2+iY}\frac{\zeta'}{\zeta}(s+w)\Gamma(w)Y^{w}\d w+O(1).    
\end{align*}

Then we move the integral to the line $\Re(w)=-1/(2(\log \log T)^{1/2})$. We pick up a residue from the pole at $w=0$ and estimate the error term from this change of lines using~\eqref{eq:gamma} and Lemma~\ref{le_logderzetalarge}.
This yields
\begin{align*}
\sum_{n\geq 1}\Lambda(n)e^{-n/Y}n^{-s}=-\frac{\zeta'}{\zeta}(s)-\frac{1}{2\pi i} \int_{-1/(2(\log \log T)^{1/2})-iY}^{-1/(2(\log \log T)^{1/2})+iY}\frac{\zeta'}{\zeta}(s+w)\Gamma(w)Y^{w}\d w+O(1).    
\end{align*}
Using again Lemma~\ref{le_logderzetalarge} and \eqref{eq:gamma}, we can bound the integral here to obtain
\begin{align}\label{eq:lambda}
\sum_{n\geq 1}\Lambda(n)e^{-n/Y}n^{-s}=-\frac{\zeta'}{\zeta}(s)+O(Y^{-1/(2(\log \log T))^{1/2}}(\log T)(\log \log T)^{1/2}).    
\end{align}
The error term here is certainly $O((\log T)/(\log \log T)^{200})$ by our choice of $Y$. Recall that~\eqref{eq:lambda} holds for $s = b+it$ for any $b \in [\sigma, 2]$. Hence, taking $s=b+it$ in \eqref{eq:lambda} and integrating both sides from $b=\sigma$ to $b=2$, we obtain
\begin{align*}
\log \zeta(\sigma+it)-\log \zeta(2+it)=\sum_{n\geq 2}\frac{\Lambda(n)}{\log n}e^{-n/Y}(n^{-\sigma-it}-n^{-2-it})+O\left(\frac{\log T}{(\log \log T)^{200}}\right).  
\end{align*}
We have $\log \zeta(2+it)\asymp 1$, and we can crudely estimate the sum over $n$ with absolute values to obtain
\begin{align*}
\log \zeta(\sigma+it)=O\left(\frac{Y^{1-\sigma}}{\log Y}+\frac{\log T}{(\log \log T)^{200}}\right).    
\end{align*}
Recalling the choice of $Y$ this means that $|\zeta(\sigma+it)|\leq T^{1/(\log \log T)^{100}}$ if $T\geq T_0$ with $T_0$ large enough. This completes the proof.
\end{proof}

\section{From large zeta values to large Dirichlet polynomial values --- proof of Lemma~\ref{le_zetabig->Rbig}}
We now prove Lemma~\ref{le_zetabig->Rbig} using the approximate functional equation for the zeta function (Lemma~\ref{le:AFE}).
\begin{proof}[Proof of Lemma~\ref{le_zetabig->Rbig}]
Writing
\[
\mathcal{T} \coloneqq \{t \in [-T, T] \colon |\zeta(\sigma+it)| \geq T^\beta\} \setminus [-2\pi, 2\pi],
\]
we need to show that
\begin{equation}
\label{eq:afeconcl}
|\mathcal{T}| \leq R_{\sigma+(2-\varepsilon)\beta, \varepsilon\beta/3}(2T).
\end{equation}
By the approximate functional equation for $\zeta(s)$ (Lemma~\ref{le:AFE} with $x = y = |t/2\pi|^{1/2}$), we have
\begin{align*}
\zeta(\sigma+it)=\sum_{n\leq |t/(2\pi)|^{1/2}} \frac{1}{n^{\sigma+it}}+ \chi(\sigma+it)\sum_{n\leq |t/(2\pi)|^{1/2}} \frac{1}{n^{1-\sigma-it}}+O(1),  \end{align*}
From the definition of $\mathcal{T}$ and~\eqref{eq:chibound}, we see that, once $T_0$ is sufficiently large in terms of $\eta_0$, for each $t \in \mathcal{T}$,
\begin{align*}
\left|\sum_{n \leq |t/(2\pi)|^{1/2}} \frac{1}{n^{\sigma + it}}\right|\geq \frac{1}{3} T^{\beta}\textnormal{ or }  \left|\sum_{n \leq |t/(2\pi)|^{1/2}} \frac{1}{n^{1-\sigma-it}}\right| \geq \frac{1}{300} T^{\beta+\sigma-1/2}.
\end{align*}
By Lemma \ref{le_partialsummation} and dyadic decomposition, this in turn implies that for some $1\leq M\leq T^{1/2}/2$ and some constant $c>0$ we have (once $T_0$ is sufficiently large in terms of $\beta_0$ and $\varepsilon_0$)
\begin{align*}
\left|\sum_{M < m \leq 2M} \frac{1}{n^{\sigma + (2-\varepsilon)\beta + it}}\right| &\geq c \frac{T^\beta M^{-(2-\varepsilon)\beta}}{\log T} \geq T^{\beta \varepsilon/3}\\
\textnormal{or }\qquad \left|\sum_{M < m \leq 2M} \frac{1}{n^{\sigma + (2-\varepsilon)\beta+it}}\right|&\geq c \frac{T^{\beta+\sigma-1/2} M^{1-2\sigma-(2-\varepsilon)\beta}}{\log T} \geq T^{\beta \varepsilon/3}.
\end{align*}
Since this holds for every $t\in \mathcal{T}$, we deduce that~\eqref{eq:afeconcl} indeed holds.
\end{proof}

\section{The inductive argument --- proof of Lemma~\ref{le_density2}}
We are now ready to prove Lemma~\ref{le_density2}.
\begin{proof}[Proof of Lemma~\ref{le_density2}]
Fix $\varepsilon, \eta > 0$. We can clearly assume that $\varepsilon$ and $\eta$ are small.  We shall use induction on $j \in \{0, 1, \dotsc, J\}$, where $J \coloneqq \lceil \frac{1/2}{\varepsilon/100}\rceil$. Write $\eta_0 = \eta (\varepsilon/20)^J$, and let $T_1 = T_1(\varepsilon, \eta)$ be sufficiently small in terms of $\varepsilon, \eta$, in particular we take $T_1 \geq T_0(\eta_0, \varepsilon)$ with $T_0$ as in Lemma~\ref{le:RsigmaRec}.

Our induction claim is that
\begin{equation}
\label{eq:IndClaim}
R_{\sigma,\eta'}(T)\leq 1000^j (\log (4^jT))^{4 j} T^{(1-\sigma)(\frac{1}{2}+\varepsilon)+\varepsilon/2}  \max_{\sigma-\varepsilon/2 \leq \alpha \leq 1} T^{(\alpha-1)(\frac{1}{2}+\varepsilon)} (N(\alpha, 4^j T)+1)
\end{equation}
whenever $\sigma \geq 1-j/(2J)$ and $\eta' \geq \eta_0 (\varepsilon/20)^{-j}$ and $T \geq T_1(\eta, \varepsilon)$. Note that this claim for $j = J$ implies Lemma~\ref{le_density2} (the restriction on $T$ is not a problem since the implied constant in Lemma~\ref{eq:indconcl} is allowed to depend on $\eta, \varepsilon$).

The base case $j=0$ is trivial since for $\sigma \geq 1$ and $\eta' \geq \eta(\varepsilon/20)^J$ we have $R_{\sigma, \eta'} = 0$ once $T_1$ is sufficiently large in terms of $\eta$ and $\varepsilon$.

Assume now that the induction claim holds for $j-1$ for some $j \in \{1, 2, \dotsc, J\}$. We shall show that it holds for $j$. Let now $\sigma \geq 1-j/(2J), \eta' \geq \eta_0 (\varepsilon/20)^{-j}$, and $T \geq T_1(\eta, \varepsilon)$, and let $\beta$ be as in Lemma~\ref{le:RsigmaRec} with $\eta'$ in place of $\eta$.

If $\beta \leq \varepsilon / 3$, then by \eqref{eq:R3} we have\footnote{We note that iterating \eqref{eq:R2} arbitrarily many times is not sufficient to reach the case $j=J$ --- at every step of the iteration the value of $\beta$ might be $\varepsilon/4$ times the previous value, and $\sum_{j\geq 1}\varepsilon^j$ converges. Hence, we eventually need to apply \eqref{eq:R3} as well --- as one would expect, since it is the part referring to zero density results.}, 
\begin{align*}
R_{\sigma,\eta'}(T)&\leq 50 T^{\beta}(\log T)^{4}(N(\sigma-1/(\log \log T)^{1/2},4T)+1) + 4 T^{\frac{1-\sigma}{2}} \\
&\leq 50 T^{\varepsilon/3} (\log T)^{4} (N(\sigma-1/(\log \log T)^{1/2}, 4 T) + 1) +  4T^{\frac{1-\sigma}{2}},
\end{align*}
which is sufficient once $T_1$ is sufficiently large.

On the other hand, if $\beta > \varepsilon/ 3$, then by \eqref{eq:R2}, 
\begin{align}
\label{eq:Rrec}
R_{\sigma,\eta'}(T)\leq 50 T^{\beta}(\log T)^{2} R_{\sigma+(2-\varepsilon)\beta+1/\log T,\varepsilon \beta/4}(2T) + 4 T^{\frac{1-\sigma}{2}}.
\end{align}
Now 
\[
\sigma+(2-\varepsilon)\beta+\frac{1}{\log T} \geq 1-\frac{j}{2J} + (2-\varepsilon)\frac{\varepsilon}{3} \geq 1-\frac{j-1}{2J}
\]
and 
\[
\frac{\varepsilon \beta}{4} \geq \frac{\varepsilon \eta'}{8} \geq \frac{\varepsilon}{8} \cdot \left(\frac{\varepsilon}{20}\right)^{-j} \eta_0 \geq \left(\frac{\varepsilon}{20}\right)^{-(j-1)} \eta_0.
\]
Hence, we can apply the induction hypothesis~\eqref{eq:IndClaim} for $j-1$ to the right-hand side of~\eqref{eq:Rrec}, obtaining
\begin{align*}
&R_{\sigma,\eta'}(T) \leq 50 T^{\beta}(\log T)^{2} \cdot 1000^{j-1} (\log (4^j T))^{4 (j-1)} (4T)^{(1-\sigma-(2-\varepsilon)\beta-1/\log T)(1/2+\varepsilon)+\varepsilon/2} \\
&\qquad \cdot \max_{\sigma+(2-\varepsilon)\beta+1/\log T-\varepsilon/2 \leq \alpha \leq 1} (4T)^{(\alpha-1)(1/2+\varepsilon)} (N(\alpha, 4^{j-1} \cdot 2 T)+1) + 4 T^{\frac{1-\sigma}{2}} \\
&\leq \frac{1000^j}{4} (\log (4^j T))^{4(j-1)+2}  T^{\beta+(1-\sigma-(2-\varepsilon)\beta)(1/2+\varepsilon)+\varepsilon/2} \max_{\sigma \leq \alpha \leq 1} T^{(\alpha-1)(1/2+\varepsilon)} (N(\alpha, 4^j T)+1) + 4 T^{\frac{1-\sigma}{2}} \\
&\leq 1000^j (\log (4^j T))^{4j}  T^{(1-\sigma)(1/2+\varepsilon)+\varepsilon/2} \max_{\sigma \leq \alpha \leq 1} T^{(\alpha-1)(1/2+\varepsilon)} (N(\alpha, 4^j T)+1)
\end{align*}
once $T_1$ is sufficiently large. Thus~\eqref{eq:IndClaim} holds for $j$ as claimed.  
\end{proof}

\section{Proof of Theorem~\ref{thm_DH}}\label{sec:thm_DH}

We need the following proposition in the proof of Theorem~\ref{thm_DH}.

\begin{proposition}\label{prop_sigma1/2} Assume Conjecture~\ref{conj:StrongDH}.
Let $\varepsilon>0$ and $k\in \mathbb{N}$. Let $T$ be sufficiently large in terms of $\varepsilon$. Let $F(s) = \prod_{j = 1}^k M_j(s)$, where
\[
M_j(s) = \sum_{\substack{M_j<m\leq 2M_j}} \frac{1}{m^s},
\]
with $M_j \geq T^\varepsilon$ such that $M_1 \dotsm M_k \asymp T^{1+o(1)}$.

Let $\mathcal{U} \subset [-T, T]$ be such that, for each $t \in \mathcal{U}$, one has $|M_j(1+it)| \geq M_j^{-1/2+20\varepsilon}$ for some $1\leq j\leq k$.
Then
\begin{align}\label{eq_m1m2}
\int_{\mathcal{U}} |F(1+it)|^2 \d t \ll_{\varepsilon} T^{-\delta_0}
\end{align}
for some $\delta_0 = \delta_0(\varepsilon)$.
\end{proposition}

\begin{proof}[Proof of Proposition \ref{prop_sigma1/2}] 
For given $\sigma_j$ for $j = 1, \dotsc, k$, write 
\begin{align*}
\mathcal{T} = \{t\in [-T,T] \colon |M_j(1+it)|\in (M_j^{-\sigma_j},2M_j^{-\sigma_j}] \text{ for every $j \in \{1, \dotsc, k\}$}\}.
\end{align*}

It clearly suffices to show that, for some $\delta_1 = \delta_1(\varepsilon) > 0$,
\begin{align}
\label{eq:Tclaim}
|\mathcal{T}| \ll_\varepsilon T^{2\min_j \sigma_j - \delta_1}
\end{align}
for every choice of $\sigma_1, \dotsc, \sigma_k$.

Without loss of generality, we can assume that $\sigma_1=\min_j \sigma_j$. By assumption $\sigma_1 \leq 1/2-20\varepsilon$. On the other hand, by a pointwise estimate for zeta sums (see e.g.~\cite[Corollary 8.26]{iw-kow}) there exists an absolute constant $\beta \in (0, 1)$ such that
\[
|M_j(1+it)| \ll \exp(-\beta (\log M_j)^3/(\log T)^2)) \ll M_j^{-\beta \varepsilon^2}.
\]
Hence, taking $\varepsilon_0 = \beta \varepsilon^2/2$ we see that whenever $\sigma_1 < \varepsilon_0$ and $T$ is sufficiently large in terms of $\varepsilon$, the set $\mathcal{T}$ is empty. 

Hence we can assume that $\sigma_1 \in [\varepsilon_0, 1/2-20\varepsilon]$. Let $\delta(\varepsilon_0)$ be as in Conjecture~\ref{conj:StrongDH}. We apply first Theorem~\ref{th:density->large} with $\varepsilon = \varepsilon' := \varepsilon_0 \delta(\varepsilon_0)/6,$ and then Conjecture~\ref{conj:StrongDH}, obtaining
\begin{align*}
|\mathcal{T}| &\ll T^{\varepsilon'} \max_{1-\sigma_1-\varepsilon' \leq \alpha \leq 1} T^{\frac{\alpha-(1-\sigma_1)}{2}} N(\alpha, C \cdot T) + T^{\sigma_1/2 + \varepsilon'} \\
&\ll T^{\varepsilon'} \max_{1-\sigma_1-\varepsilon' \leq \alpha \leq 1} T^{\frac{\alpha-(1-\sigma_1)}{2}} T^{(2-\delta(\varepsilon_0))(1-\alpha)} + T^{\sigma_1/2 + \varepsilon'} \\
\end{align*}
The maximum is attained for $\alpha = 1-\sigma_1-\varepsilon'$. Recalling that $\sigma_1 \geq \varepsilon_0$, we obtain
\begin{align*}
|\mathcal{T}| &\ll T^{\varepsilon'} T^{-\frac{\varepsilon'}{2}} T^{(2-\delta(\varepsilon_0))(\sigma_1+\varepsilon')} + T^{\sigma_1/2 + \varepsilon'} \\
&\ll T^{2\sigma_1 + 3\varepsilon' - \delta(\varepsilon_0) \varepsilon_0} + T^{2\sigma_1 - \frac{3\varepsilon_0}{2} + \varepsilon'}.
\end{align*}
Recalling the definition of $\varepsilon'$, the claim~\eqref{eq:Tclaim} follows with $\delta_1 = \delta(\varepsilon_0) \varepsilon_0/2$. 
\end{proof}

We are now ready to sketch the proof of Theorem~\ref{thm_DH}: The proof follows similarly to the same theorem under the Lindel\"of hypothesis in~\cite[Section 6]{matomaki-teravainenE2} --- there the Lindel\"of hypothesis was only used to deduce that a condition similar to $|M_j(1+it)| \geq M_j^{-1/2+20\varepsilon}$ never holds. Assuming Conjecture~\ref{conj:StrongDH}, we can instead use Proposition~\ref{prop_sigma1/2} to deal with the contribution of the case when this does hold for some $j$.

\section{Proof of Proposition~\ref{prop:converse}}
\label{se:converse}
We now prove Proposition~\ref{prop:converse} by refining a standard zero-detecting polynomial.
\begin{proof}[Proof of Proposition~\ref{prop:converse}]
By dyadic splitting, it suffices to show that, for any $U \leq T/2$, we can partition
\[
\mathcal{U} \coloneqq \{\rho = \beta + i \gamma \colon \zeta(\rho) = 0, \beta \geq 1-\nu, U \leq |\gamma| \leq 2U\}
\] 
into two sets, $\mathcal{U} = \mathcal{U}_1 \cup \mathcal{U}_2$ in such a way that $\# \mathcal{U}_1 \ll T^{2\nu+3\varepsilon/2}$ and the zeroes in $\mathcal{U}_2$ satisfy the same condition as requested from the zeroes in $\mathcal{T}_2$ in the statement of the proposition. The case $U \leq T^\varepsilon$ follows trivially by simply choosing $\mathcal{U}_1 = \mathcal{U}$ and applying Lemma~\ref{le:zerosinboxes}. Similarly we can assume that $T$ is sufficiently large in terms of $\varepsilon$.

From now on we assume that $U \in [T^\varepsilon, T/2]$. We use the same standard zero-detecting polynomial as in~\cite[Section 10.2]{iw-kow} (with different length). We write $R \coloneqq T^{\varepsilon^2/2}$ and 
\[
R(s) := \sum_{r \leq R} \frac{\mu(r)}{r^s}.
\]
Let $\rho = \beta + i\gamma$ be a zeta zero with $\beta \geq 1-\nu$ and $|\gamma| \in [U, 2U]$. From Lemma~\ref{le:AFE}(ii) and the estimate $R(\beta + i\gamma) \ll R^{1-\beta} \log R$ we see that
\[
0 = \zeta(\beta + i\gamma) R(\beta + i\gamma) = \sum_{n \leq U R} \frac{a_n}{n^{\beta+i\gamma}} + O(U^{-\beta} R^{1-\beta} \log R),
\]
where
\[
a_n = \sum_{\substack{\ell r = n \\ \ell \leq U, \, r \leq R}} \mu(r).
\]
Now $a_1 = 1$ and $a_n = 0$ for $1 < n \leq R$. Hence, once $T$ is sufficiently large, 
\[
\left|\sum_{\substack{R < \ell r \leq UR \\ \ell \leq U, \, r \leq R}} \frac{\mu(r)}{(\ell r)^{\beta+i\gamma}}\right| \geq \frac{1}{2}.
\]
Splitting dyadically, we see that, for each zeta zero $\rho$, there exists $K_\rho = 2^{k_\rho} R \in [R, UR]$ such that
\begin{equation}
\label{eq:dyadlower}
\left|\sum_{\substack{K_\rho < \ell r \leq 2K_\rho \\ \ell \leq U, \, r \leq R}} \frac{\mu(r)}{(\ell r)^{\beta+i\gamma}}\right| \geq \frac{1}{4\log T}.
\end{equation}
We choose
\[
\mathcal{U}_1 \coloneqq \{\rho \in \mathcal{U} \colon K_\rho \leq R T^\varepsilon \text{ or } K_\rho \geq U T^{-\varepsilon}\} \quad \text{and} \quad \mathcal{U}_2 \coloneqq \mathcal{U} \setminus \mathcal{U}_1.  
\]
Let us first show that $\#\mathcal{U}_1 \ll T^{2\nu+3\varepsilon/2}$. By the pigeonhole principle, for some $K' \in [R, RT^\varepsilon] \cup [U T^{-\varepsilon}, UR]$ we have 
\[
\# \mathcal{U}_1 \leq 2 \log T \cdot \#\left\{\rho = \beta + i\gamma \in \mathcal{U}_1 \colon \left|\sum_{\substack{K' < \ell r \leq 2K' \\ \ell \leq U, \, r \leq R}} \frac{\mu(r)}{(\ell r)^{\beta+i\gamma}}\right| \geq \frac{1}{4\log T}\right\}.
\]
By Lemma~\ref{le:zerosinboxes} we can find a one-spaced set $\mathcal{V}_1 \subseteq \mathcal{U}_1$ such that $\#\mathcal{V}_1 \gg \#\mathcal{U}_1/\log^2 T$ and for each $\rho = \beta + i\gamma \in \mathcal{V}_1$, we have, with $\sigma = \beta-(1-\nu) \geq 0$, 
\[
\left|\sum_{\substack{K' < \ell r \leq 2K' \\ \ell \leq U, \, r \leq R}} \frac{\mu(r) (\ell r)^{-\sigma}}{(\ell r)^{1-\nu+i \gamma}}\right| \geq \frac{1}{4\log T}.
\]
By the discrete mean value theorem applied to the $2k$th moment of the Dirichlet polynomial on the left hand side (see e.g.~\cite[Proposition 9.11]{iw-kow}), we obtain that
\begin{equation}
\label{eq:powerDir}
\#\mathcal{V}_1 \ll_\varepsilon (U+K'^k) \sum_{K'^k < n \leq (2K')^k} \frac{d_{2k}(n)^2}{n^{2(1-\nu)}} (\log T)^{O_k(1)} \ll_k (U+K'^k) K'^{2\nu-1} (\log T)^{O_{\varepsilon,k}(1)}
\end{equation}
Choosing $k = \lfloor \frac{\log U}{\log K'}\rfloor$ if $K' \in [R, R T^\varepsilon]$ and $k=1$ if $K' \in [UT^{-\varepsilon}, UR]$, we obtain that $\#\mathcal{V}_1 \ll_{\varepsilon} U^{2\nu + 5\varepsilon/4}$, and thus $\#\mathcal{U}_1 \ll_{\varepsilon} U^{2\nu+3\varepsilon/2}$.

Consider then the set $\mathcal{U}_2 = \mathcal{U} \setminus \mathcal{U}_1$. By~\eqref{eq:dyadlower} and the triangle inequality, for each $\rho = \beta + i\gamma \in \mathcal{U}_2$, there exists $K_\rho \in (RT^\varepsilon, U T^{-\varepsilon}]$ such that
\[
\sum_{r \leq R} \frac{1}{r^\beta} \left|\sum_{\substack{K_\rho/r < \ell \leq 2K_\rho/r}} \frac{1}{\ell^{\beta+i\gamma}}\right| \geq \frac{1}{4\log T}.
\]
By the pigeonhole principle there exists $1\leq r_\rho \leq R$ such that
\begin{equation}
\label{eq:ellsumlow}
\left|\sum_{\substack{K_\rho/r_\rho < \ell \leq 2K_\rho/r_\rho}} \frac{1}{\ell^{\beta+i\gamma}}\right| \geq \frac{r_\rho^{\beta-1}}{8\log^2 T} \geq 1600 \cdot T^{-\varepsilon^2/2}
\end{equation}
once $T$ is sufficiently large in terms of $\varepsilon$.

We split into two cases according to the size of $K_\rho/r_\rho$.

\textbf{Case 1: $K_\rho/r_\rho \leq T^{1/2}/2$.} Now~\eqref{eq:ellsumlow} holds with some $K_\rho/r_\rho \in [T^\varepsilon, T^{1/2}/2]$, so by Lemma~\ref{le_partialsummation} there exists $y \in [T^\varepsilon, T^{1/2}/2]$ and $y' \in (y, 2y]$ such that
\[
\left|\sum_{\substack{y < \ell \leq y'}} \frac{1}{\ell^{1+i\gamma}}\right| \geq 200 \cdot T^{-\varepsilon^2/2} y^{\beta-1}. 
\]
Since $y \geq T^\varepsilon$, this is $\geq y^{-\nu-\varepsilon}$ and the claim follows. 

\textbf{Case 2: $K_\rho/r_\rho > T^{1/2}/2$.} Now~\eqref{eq:ellsumlow} holds with $K_\rho/r_\rho \in [T^{1/2}/2, UT^{-\varepsilon}]$. In particular $U \geq T^{1/2}$ in this case. Applying Lemma~\ref{le:AFE} with $x = 2K_\rho/r_\rho$ and $x = K_\rho/r_\rho$ implies that
\[
\sum_{\substack{\frac{K_\rho}{m_\rho} < \ell \leq \frac{2K_\rho}{m_\rho}}} \frac{1}{\ell^{\beta+i\gamma}} = \chi(\beta+i\gamma) \sum_{\frac{|\gamma|}{4\pi K_\rho/r_\rho} < n \leq \frac{|\gamma|}{2\pi K_\rho/r_\rho}} \frac{1}{n^{1-\beta-i\gamma}} + O\left(\frac{\log U}{(K_\rho/r_\rho)^{\beta}} + \frac{(K_\rho/r_\rho)^{1-\beta}}{U^{1/2}}\right).
\]
Hence by~\eqref{eq:ellsumlow} and~\eqref{eq:chibound} we can find $y \in [T^\varepsilon, U/(2T^{1/2})]$ such that
\[
\left| \sum_{y < n \leq 2y} \frac{1}{n^{1-\beta-i\gamma}} \right| \geq 8 \cdot U^{\beta-1/2} T^{-\varepsilon^2/2}.
\]
By Lemma~\ref{le_partialsummation} this implies that there exists $y \in [T^\varepsilon, U/(2T^{1/2})]$ and $y' \in (y, 2y]$ such that
\[
\left|\sum_{\substack{y < \ell \leq y'}} \frac{1}{\ell^{1+i\gamma}}\right| \geq \frac{ U^{\beta-1/2} T^{-\varepsilon^2/2}}{y^{\beta}}.
\]
Since $y \leq U$, we obtain a lower bound by applying the inequality $\beta \geq 1-\nu$, and so
\begin{equation}
\label{eq:case2fin}
\left|\sum_{\substack{y < \ell \leq y'}} \frac{1}{\ell^{1+i\gamma}}\right| \geq \frac{U^{1/2-\nu} T^{-\varepsilon^2/2}}{y^{1-\nu}} = y^{-\nu} T^{-\varepsilon^2/2} \left(\frac{U}{y^2}\right)^{1/2-\nu}.
\end{equation}
Using also that $y^2 \leq U^2/(4T) \leq U/4$ and $y \geq T^\varepsilon$, we see that the right hand side of~\eqref{eq:case2fin} is $\geq y^{-\nu-\varepsilon}$ and the claim follows.
\end{proof}

\section*{Acknowledgements} 
KM was partially supported by Academy of Finland grant no. 285894. JT was supported by Academy of Finland grant no. 340098 and by funding from the European Union's Horizon Europe research and innovation programme under Marie Sk\l{}odowska-Curie grant agreement no. 101058904.

Part of this material is based upon work supported by the Swedish Research Council under grant no. 2021-06594 while the authors were in residence at Institut Mittag-Leffler in Djursholm, Sweden, during Spring 2024.

\bibliography{refs}
\bibliographystyle{plain}

\end{document}